\documentclass[letterpaper, 10 pt, conference]{ieeeconf}  

\IEEEoverridecommandlockouts                              

\usepackage{cite}

\usepackage{amsmath,amssymb,amsfonts}
\usepackage{graphicx}
\usepackage{textcomp}
\usepackage{mathtools}
\usepackage{xcolor}
 
\usepackage{amsthm}
\usepackage{caption}
\usepackage{fancyhdr}
\usepackage{algorithm}
\usepackage[]{algpseudocode}
\def\BibTeX{{\rm B\kern-.05em{\sc i\kern-.025em b}\kern-.08em
    T\kern-.1667em\lower.7ex\hbox{E}\kern-.125emX}}
    
\usepackage{epstopdf}

\newtheorem{theorem}{Theorem}

\newtheorem{lemma}{Lemma}
\newtheorem{definition}{Definition}
\newtheorem{remark}{Remark}
\newtheorem{assum}{Assumption}

\title{\LARGE \bf
A stochastic output-feedback MPC scheme for distributed systems
}

\author{Christoph Mark and Steven Liu
\thanks{Institute  of  Control  Systems,  Department  of  Electrical  and  Computer
Engineering, University of Kaiserslautern, Erwin-Schrödinger-Str. 12, 67663
Kaiserslautern, Germany, {\tt\small [mark|sliu]@eit.uni-kl.de}}%
}

\fancypagestyle{FirstPage}{
	\lfoot{		\begin{minipage}{\textwidth}
			\footnotesize
		This version has been accepted for publication in Proc. American Control Conference (ACC), 2020. Personal use of this material is permitted. Permission from AACC must be obtained for all other uses, in any current or future media, including reprinting/republishing this material for advertising or promotional purposes, creating new collective works, for resale or redistribution to servers or lists, or reuse of any copyrighted component of this work in other works.
	\end{minipage}} 
}

\begin{document}

\maketitle
\thispagestyle{empty}
\pagestyle{empty}

\begin{abstract}
In this paper, we present a novel stochastic output-feedback MPC scheme for distributed systems with additive process and measurement noise. The chance constraints are treated with the concept of probabilistic reachable sets, which, under an unimodality assumption on the disturbance distributions are guaranteed to be satisfied in closed-loop. By conditioning the initial state of the optimization problem on feasibility, the fundamental property of recursive feasibility is ensured. Closed-loop chance constraint satisfaction, recursive feasibility and convergence to an asymptotic average cost bound are proven. The paper closes with a numerical example of three interconnected subsystems, highlighting the chance constraint satisfaction and average cost compared to a centralized setting.
\end{abstract}


\section{Introduction}
\thispagestyle{FirstPage}
Model Predictive Control (MPC) is in its standard form a full state-feedback control strategy \cite{mpc}. However, this limits its applicability in many practical situations where the state vector is commonly not fully measurable and only an estimate of the true state is available, which leads to the output-feedback MPC framework \cite{Findeisen} \cite{Mayne}.

If uncertainties are present, the literature of MPC is separated into robust \cite{bemporad1} \cite{mayne2} and stochastic approaches \cite{mesbah}. The difference between them is in general that in stochastic MPC (SMPC) the underlying distribution of the disturbance is taken into account, while in robust MPC a bounded worst-case disturbance is considered. Therefore, in SMPC the hard constraints are relaxed to hold probabilistically as chance constraints. SMPC is distinguished in two kind of approaches. The fist one is called \textit{randomized approach} \cite{Prandini} \cite{Schildbach}, where at every time step a sufficient number of disturbance realizations is sampled in order to find a optimal input sequence to the system. These methods can deal with arbitrary disturbance realizations but their heavy computational load is a tough bottleneck for fast online implementations. The second method is based on analytical approximations of the stochastic control problem, namely \textit{probabilistic approximation method} \cite{Cannon1}  \cite{Farina1} \cite{Mark2}.
 
While the vast majority of the stochastic MPC approaches is developed for centralized setups, only a few methods are concerned about an efficient distributed implementation. Furthermore, to the best of the author's knowledge, the stochastic output-feedback case for distributed systems has never been investigated in view of closed-loop chance constraint satisfaction, nor with an iterative controller structure. These issues were recently highlighted as an open research direction \cite{mesbah}. The necessity for distributed MPC strategies is emerging due to the increasing complexity of the underlying control systems \cite{Christofides} \cite{Maestre}.
\subsection*{Related work}
In \cite{Cannon1} the concept of stochastic tubes was introduced, which was later on extended to the output-feedback case \cite{Cannon3}. In \cite{Kouvaritakis} the stochastic tube concept was further extended to probabilistic tubes, whereas in \cite{Lorenzen} a general constraint tightening framework was presented, both leading to a less conservative feasible region. These approaches rely on a boundedness assumption of the underlying disturbance distribution and were developed for central MPC setups.

In \cite{Cannon2} \cite{Farina1} \cite{Paulson} \cite{b10} the boundedness assumption was relaxed to infinite support. Hence, recursive feasibility cannot be achieved by constraint tightening. These approaches typically rely on a backup solution, which is applied whenever the problem becomes infeasible. In \cite{Hewing2} the authors proposed a strictly recursive feasible SMPC based on indirect feedback.

In \cite{Farina2} an output-feedback stochastic MPC scheme is presented, which extends the state-feedback formulation from \cite{Farina1}. In both approaches the chance constraints are reformulated via Cantelli's inequality. The main drawback of this approach is the lack of closed-loop guarantees.

\subsection*{Contribution}
In this paper, we develop a stochastic output-feedback MPC scheme for distributed systems. The underlying MPC optimization problem is reduced to a quadratic program, which we opt to solve via distributed optimization. The chance constraints are treated with the concept of probabilistic reachable sets (PRS) \cite{b10}, which we recently proposed to use in a distributed setting \cite{Mark1}. We extend the distributed PRS concept to the output-feedback case, such that the synthesis of distributed PRS can be done fully parallelizable via distributed optimization. The MPC algorithm is proven to be recursively feasible with guaranteed closed-loop chance constraint satisfaction and asymptotic convergence to an average cost bound. Since we solve the MPC problem via distributed optimization, we do not rely on an initially known central state and input sequence to initialize the controllers. Hence, the controller synthesis and the closed-loop operation do not need a central coordination node.  
\subsection*{Outline}
The first section introduces the notations and the problem setup. The second section is dedicated to the controller structure, where afterwards the estimation and prediction errors are reformulated for a joint computation of the covariance prediction based on linear matrix inequalities (LMI) \cite{b8}. The section continues with the chance constraint tightening, the introduction of the cost functions and the global MPC optimization problem. The section ends with the main result on recursive feasibility, closed-loop chance constraint satisfaction and convergence. The paper closes with an example of the proposed approach and some concluding remarks.

\section{Preliminaries}
\label{sec:preliminaries}
\subsection{Notations}
Given two polytopic sets $\mathbb{A}$ and $\mathbb{B}$, the Pontryagin difference is given as $\mathbb{A} \ominus \mathbb{B} = \{ a \in \mathbb{A} : a+b \in \mathbb{A}, \forall b \in \mathbb{B} \}$. The set of positive real numbers is defined as $\mathbb{R}_{> 0}$, whereas positive definite and semidefinite matrices are indicated as $A>0$ and $A\geq0$, respectively. Given a matrix $A$ and vector $x$, we denote the $(i,j)$-th element of $A$ as $[A]_{i,j}$ and the $j$-th element of $x$ as $[x]_j$. The spectral radius of a matrix $A$ is denoted as $\rho(A)$. The weighted 2-norm is $\Vert x \Vert_P = \sqrt{x^\top P x}$.
For an event $E$ we define the probability of occurrence as $\text{Pr}(E)$, whereas the expected value of a random variable $w$ is given by $\mathbb{E}(w)$. The set $ \{ 1, ..., M \} \subseteq \mathbb{N} $ is denoted as $\mathcal{M}$.
\subsection{Problem description}
We consider a network of $M$ linear time-invariant systems, where each system $i \in \mathcal{M}$ has a state vector $x_i \in \mathbb{R}^{n_i}$, input vector $u_i \in \mathbb{R}^{m_i} $ and output vector $y_i \in \mathbb{R}^{p_i}$. The distribution functions of the zero-mean i.i.d. process noise $w_i \in \mathbb{R}^{n_i}$ and zero-mean i.i.d. measurement noise $d_i \in \mathbb{R}^{p_i}$ are assumed to be central convex unimodal (CCU), i.e. $w_i \sim \mathcal{Q}_i^W(0, \Sigma_i^W)$ and $d_i \sim \mathcal{Q}_i^D(0,\Sigma_i^D)$, where additionally the second moments $\Sigma_i^W$ and $\Sigma_i^D$ are known. The local dynamics are governed by
\begin{equation}
\begin{aligned}
	x_i(k+1) &= \sum_{j=1}^M A_{ij} x_j(k) + B_i u_i(k) + w_i(k)\label{eq:local_dynamics_raw} \\ 
	y_i(k) &= \sum_{j=1}^M C_{ij} x_j(k) + d_i(k),
\end{aligned}
\end{equation}
where $A_{ij} \in \mathbb{R}^{n_i \times n_i}$, $C_{ij} \in \mathbb{R}^{p_i \times n_j}$ and $B_{i} \in \mathbb{R}^{m_i \times n_i}$.
The local states and inputs are constrained in convex polytopes, which contain the origin in their interior
\begin{align*}
\mathbb{X}_i = \{x_i | H^x_i x_i \leq h^x_i\}, \:\mathbb{U}_i = \{u _i| H^u_i u_i \leq h^u_i\} \: \forall i \in \mathcal{M}, 
\end{align*}
where afterwards the stochasticity of the problem is utilized to formulate point-wise in-time chance constraints 
\begin{subequations}
\begin{align}
&\text{Pr}(x_i(k) \in \mathbb{X}_i) \geq p_{i,x} \: \forall k\geq 0\\
&\text{Pr}(u_i(k) \in \mathbb{U}_i) \geq p_{i,u} \: \forall k\geq 0.
\end{align} \label{eq:chance_constraints}
\end{subequations}
The constants $p_{i,x} \in (0,1)$ and $p_{i,u} \in (0,1)$ are the probability levels of constraint satisfaction for states and inputs for each subsystem $i\in \mathcal{M}$. Similar to \cite{Conte1} we express the coupling dynamics with the notion of neighboring systems.
\begin{definition}[\hspace{1sp}\cite{Conte1} Neighboring systems]
System $j$ is a neighbor of system $i$ if $A_{ij} \neq 0$ or $C_{ij} \neq 0$. The set of all neighbors of system $i$, including system $i$ itself, is denoted as $\mathcal{N}_i$. The states of all systems $j \in \mathcal{N}_i$ are denoted as $x_{\mathcal{N}_i} \in \text{col}_{j \in \mathcal{N}_i}(x_j) \in \mathbb{R}^{n_{\mathcal{N}_i}}$. 
\end{definition}
The local dynamics \eqref{eq:local_dynamics_raw} can be written compactly as
\begin{subequations}
\begin{align}
		x_i(k+1) &= A_{\mathcal{N}_i} x_{\mathcal{N}_i}(k) + B_i u_i(k) + w_i(k) \label{eq:local_dynamics_neighbor} \\
		y_i(k) &= C_{\mathcal{N}_i} x_{\mathcal{N}_i}(k) + d_i(k), \label{eq:local_output_neighbor}
\end{align}
\label{eq:local_system}
\end{subequations}
whereas the global dynamics are given by
\begin{align}
	\begin{aligned}
		x(k+1) &= A x(k) + B u(k) + w(k) \\
		y(k) &= C x(k) + d(k),	
	\end{aligned}
	 \label{eq:global_system}
\end{align}
with $x = \text{col}_{i \in \mathcal{M}}(x_i)$, $u = \text{col}_{i \in \mathcal{M}}(u_i)$, $w = \text{col}_{i \in \mathcal{M}}(w_i)$ and $d = \text{col}_{i \in \mathcal{M}}(d_i)$. From \eqref{eq:local_dynamics_raw} we have that $A\in \mathbb{R}^{n \times n}$ and $C \in \mathbb{R}^{p \times n}$ are block-sparse and $B \in \mathbb{R}^{n \times m}$ is block diagonal.
\begin{assum}(Structured controller and injection gain)~
	\begin{itemize}
	\item The pair $(A,B)$ is stabilizable with a structured linear feedback control law of the form
\begin{align*}
\kappa(x) = Kx = \text{col}_{i \in \mathcal{M}}(K_{\mathcal{N}_i} x_{\mathcal{N}_i}),
\end{align*}	
where $K_{\mathcal{N}_i} \in \mathbb{R}^{m_i \times n_{\mathcal{N}_i}}$, such that $\rho(A + BK) < 1$.
 \item The pair $(A,C$) is observable with a structured linear injection gain of the form
\begin{align*}
\lambda(y) = L y = \text{col}_{i \in \mathcal{M}}(L_i y_i)
\end{align*}	
where $L_{i} \in \mathbb{R}^{n_i \times p_i}$, such that $\rho(A-LC) < 1$
	\end{itemize}
\label{assum_stabilizable}
\end{assum}
\begin{remark}
\label{remark_stabilizable}
The structured controllers $K_{\mathcal{N}_i}$ can be computed via structured LMIs, e.g. \cite[Lemma 10]{Conte1}. By setting $(A_{\mathcal{N}_i}, B_i) = (A_{\mathcal{N}_i}^\top, C_{\mathcal{N}_i}^\top)$, the structured injection gains $L_{i}$ can similarly be derived.
\end{remark}

\section{Distributed Output feedback SMPC}
In this paper, we aim to design an iterative distributed MPC algorithm based on output-feedback for system \eqref{eq:local_dynamics_raw}. Given \eqref{eq:local_system}, for each subsystem $i \in \mathcal{M}$ we define a distributed Luenberger observer, which provides an estimate $\hat{x}_i$ of the real state $x_i$ based on the output $y_i$
\begin{align*}
	\hat{x}_i(k+1) &= A_{\mathcal{N}_i} \hat{x}_{\mathcal{N}_i}(k) + B_i u_i(k) + L_i(y_i(k) - \hat{y}_i(k)),
\end{align*} 
where $\hat{y}_i(k) = C_{\mathcal{N}_i} \hat{x}_{\mathcal{N}_i}(k)$. Now we define the robust tube-based control law
\begin{align}
	u_i(k) = v_i(0|k) + K_{\mathcal{N}_i} (\hat{x}_{\mathcal{N}_i}(k) - z_{\mathcal{N}_i}(0|k)), \label{eq:input_disturbed}
\end{align}
with $z$ being the state of the nominal system
\begin{align*}
	z_i(t+1|k) = A_{\mathcal{N}_i} z_{\mathcal{N}_i}(t|k) + B_i v_i(t|k).
\end{align*}
The notations $z_{\mathcal{N}_i}(t|k)$ and $v_i(t|k)$ denote $t$-step ahead predictions of states and inputs, obtained as the result of an underlying MPC optimization problem solved at time step $k \geq 0$. The choice of the initial value $z_i(0|k), \forall i \in \mathcal{M}$ will be discussed later on.
Let further $\tilde{x}$ be the state estimation error and $e$ the observer error, i.e.
\begin{subequations}
\begin{align}	
\tilde{x}(k) &= x(k) - \hat{x}(k) \\
 e(k) &= \hat{x}(k) - z(0|k),
\end{align}
\label{eq:errors}
\end{subequations}
such that the real state is given by
\begin{align}
	x(k) = z(0|k) + e(k) + \tilde{x}(k). \label{eq:disturbed_state}
\end{align}
\subsection{Error dynamics}
In order to satisfy the chance constraints \eqref{eq:chance_constraints}, we have to characterize error bounds on the states and controls. In view of \eqref{eq:input_disturbed} and \eqref{eq:disturbed_state} this can be achieved in terms of $e$ and $\tilde{x}$. The corresponding predictive error dynamics of \eqref{eq:errors} are given by
\begin{align*}
	&\scalebox{0.94}{$\tilde{x}_i(t+1|k) = A_{\mathcal{N}_i,L} \tilde{x}_{\mathcal{N}_i}(t|k) + w_i(t|k) - L_{i} d_{i}(t|k)$},  \\
	&\scalebox{0.94}{$e_i(t+1|k) = A_{\mathcal{N}_i,K} e_{\mathcal{N}_i}(t|k) + L_{i} ( C_{\mathcal{N}_i} \tilde{x}_{\mathcal{N}_i}(t|k) + d_{i}(t|k) )$}, 
\end{align*}
where $A_{\mathcal{N}_i,L} = A_{\mathcal{N}_i} - L_{i} C_{\mathcal{N}_i}$ and $A_{\mathcal{N}_i,K} = A_{\mathcal{N}_i} + B_i K_{\mathcal{N}_i}$,The predictive error dynamics are coupled to the true dynamics with the following initial conditions:
\begin{align*}
	\tilde{x}(0|k) &= \tilde{x}(k), \\
	e(0|k) &= e(k).
\end{align*}
However, the predictive error dynamics can similarly be expressed with the augmented error dynamics
\begin{align}
\xi_i^+ = \Psi_{\mathcal{N}_i} \xi_{\mathcal{N}_i} + \Gamma_{i} \omega_i, \label{eq:extended_error}
\end{align}
where $\xi_i = [\tilde{x}_i^\top \: \: e_i^\top]^\top$, $\xi_{\mathcal{N}_i} = [\tilde{x}_{\mathcal{N}_i}^\top \: \: e_{\mathcal{N}_i}^\top]^\top$, $\omega_{i} = [w_i^\top \: \: d_i^\top]^\top$,
\begin{align*}
\Psi_{\mathcal{N}_i} = \begin{bmatrix}
			A_{\mathcal{N}_i,L} & 0 \\
			L_{i} C_{\mathcal{N}_i} & A_{\mathcal{N}_i,K}
\end{bmatrix}, \Gamma_{i} = \begin{bmatrix}
I & -L_{i}\\
0 & L_{i}
\end{bmatrix}.
\end{align*}
We loosened the notation by denoting the successor state with a $^+$, e.g. $\xi = \xi(t|k)$ and $\xi^+ = \xi(t+1|k)$. 
\subsection{Error propagation}
In order to probabilistically bound \eqref{eq:extended_error}, we make use of PRS, which are characterized through the mean $\mu = \mathbb{E}(\xi)$ and variance $\Sigma = \text{var}(\xi)$. Note that by a proper initialization of $x(0) = \hat{x}(0) = z(0)$ we achieve that $\mathbb{E}(\xi(0)) = 0$, which, together with the zero-mean process $\omega_{i}$ implies that $\mathbb{E}(\xi(t|k)) = 0, \forall t,k\geq 0$. Furthermore, the nominal state reduces to $z = \mathbb{E}(x)$.
\begin{remark}
\label{rem:over_approx}
The global covariance matrix $\Sigma = \mathbb{E}(\xi^+ \xi^{+,\top})$ is by definition a dense matrix, which is a tough bottle neck for a distributed implementation. To this end we introduce $\hat{\Sigma}$ as block diagonal upper bound of $\Sigma$, i.e. $\Sigma \leq \hat{\Sigma}$.
\end{remark}
Using the zero-mean property of $\xi$, the covariance propagation is given by
\begin{align}
	\Sigma_i^+ = \mathbb{E}(\xi_{\mathcal{N}_i}^+ \xi_{\mathcal{N}_i}^{+,\top}) = \Psi_{\mathcal{N}_i} \Sigma_{\mathcal{N}_i} \Psi_{\mathcal{N}_i}^\top + \Gamma_{i} \Omega_{i} \Gamma_{i}^\top, \label{eq:covariance_prediction}
\end{align}
where $\Omega_{i} = \text{diag}(\Sigma_i^w, \Sigma_i^d)$. Due to the block diagonality of $\hat{\Sigma}$, we can obtain the block diagonal neighborhood covariance matrices $\Sigma_{\mathcal{N}_i}$ via selector matrices, e.g. as in \cite[Sec. 4]{Conte1}. Moreover, $\Sigma_{\mathcal{N}_i}$ can be partitioned into sub matrices
\begin{align*}
\Sigma_{\mathcal{N}_i} = \left[
\begin{array}{c|c}
	\Sigma_{\mathcal{N}_i}^{\tilde{x}}
 & 0 \\
\hline
0 & 
	\Sigma_{\mathcal{N}_i}^{e}
\end{array}
\right],
\end{align*}
where the first block upper bounds to the covariance of $\tilde{x}_{\mathcal{N}_i}$ and the second block the covariance of $e_{\mathcal{N}_i}$. The local covariance matrices are equally defined as 
\begin{align*}
\Sigma_{i} = \left[
\begin{array}{c|c}
	\Sigma_i^{\tilde{x}}
 & 0 \\
\hline
0 & 	\Sigma_i^{e}
\end{array}
\right].
\end{align*}
\begin{remark}
\label{rem:ccu_property}
Note that CCU distributions are closed under linear transformation and convolution \cite{Dharmadhikari}. Hence, eq. \eqref{eq:covariance_prediction} preserves the CCU property of the propagated error distributions $\tilde{x}_i^+ \sim \mathcal{Q}_i(0, \Sigma_i^{\tilde{x}, +})$ and $e_i^+ \sim \mathcal{Q}_i(0, \Sigma_i^{e,+})$.
\end{remark}
By relaxing \eqref{eq:covariance_prediction} as an inequality, the propagation of the covariances can be characterized via structured LMIs, such that the optimization problem can be solved fully distributed. 
\begin{lemma}
\label{lem:schur}
The inequality version of \eqref{eq:covariance_prediction} is equivalent to the following structured LMI
\begin{align}
 \begin{bmatrix}
 	\Sigma_i^+ - \Gamma_{i} \Omega_{i} \Gamma_{i}^\top & \Psi_{\mathcal{N}_i} \Sigma_{\mathcal{N}_i} \\
 	\Sigma_{\mathcal{N}_i} \Psi_{\mathcal{N}_i}^\top & \Sigma_{\mathcal{N}_i}
 \end{bmatrix} \geq 0. \label{eq:covariance_prediction_LMI}
\end{align}
\end{lemma}
\begin{proof}
For positive definite $\Sigma_{\mathcal{N}_i}$ we can reformulate the inequality version of \eqref{eq:covariance_prediction} as 
\begin{align*}
	\Sigma_i^+ -  \Gamma_{i} \Omega_i \Gamma_{_i}^\top - \Psi_{\mathcal{N}_i} \Sigma_{\mathcal{N}_i} ( \Sigma_{\mathcal{N}_i})^{-1} \Sigma_{\mathcal{N}_i} \Psi_{\mathcal{N}_i}^\top \geq 0.
\end{align*}
Application of the Schur complement yields \eqref{eq:covariance_prediction_LMI}.
\end{proof}
In this formulation, we can obtain the stationary distribution of $\xi(t|k)$ for $t \rightarrow \infty$  by modifying \eqref{eq:covariance_prediction_LMI}, i.e.
\begin{align}
 \begin{bmatrix}
 	\Sigma_{f,i} - \Gamma_i \Omega_i \Gamma_{i}^\top & \Psi_{\mathcal{N}_i} \Sigma_{f, \mathcal{N}_{i}} \\
 	\Sigma_{f, \mathcal{N}_{i}} \Psi_{\mathcal{N}_i}^\top & \Sigma_{f, \mathcal{N}_{i}}
 \end{bmatrix} \geq 0, \label{eq:covariance_prediction_LMI_terminal}
\end{align}
and solving the following convex optimization problem
\begin{subequations}
\begin{eqnarray}
	\Sigma_{f} = &&\underset{}{\text{min}} \quad \sum_{i=1}^M \Vert \Sigma_{f,i} \Vert_F^2 \\
		&&\text{s.t.} \quad \eqref{eq:covariance_prediction_LMI_terminal}, \Sigma_{f,i}>0 \: \forall i \in \mathcal{M},
\end{eqnarray}
\label{eq:optimization_variance}
\end{subequations}
where the cost metric minimizes the Frobenius norm of the local covariance matrix. The matrix $\Sigma_f$ is the global block diagonal covariance matrix. Note that due to the distributed structure we can solve \eqref{eq:optimization_variance} with common distributed optimization techniques, e.g. the alternating direction method of multipliers (ADMM) \cite{boyd_admm}. 
\subsection{Probabilistic Reachable Sets}
Now we recall \eqref{eq:disturbed_state} and point out that we want to satisfy the chance constraints \eqref{eq:chance_constraints} for the true state $x$. Hence, we define $\delta x_i = e_i + \tilde{x}_i$, which can be expressed via \eqref{eq:extended_error} as $\delta x_i = [I \quad I] \: \: \xi_i$ with covariance
\begin{align}
	\Sigma_{f,i}^X = \begin{bmatrix}
		I & I 
	\end{bmatrix} \Sigma_{f,i} \begin{bmatrix}
		I & I 
	\end{bmatrix}^\top. \label{eq:state_variance}
\end{align}
Letting $\delta u =  K_{\mathcal{N}_i} (\hat{x}_{\mathcal{N}_i} - z_{\mathcal{N}_i}) = [0 \quad K_{\mathcal{N}_i}] \;\xi_{\mathcal{N}_i}$, then $\mathbb{E}(\delta u) = 0$ and the covariance matrix is given by
\begin{align}
	\Sigma_{f,i}^U = \begin{bmatrix}
		0 & K_{\mathcal{N}_i} 
	\end{bmatrix} \Sigma_{f, \mathcal{N}_{i}}\begin{bmatrix}
		0 & K_{\mathcal{N}_i} 
	\end{bmatrix}^\top. \label{eq:input_variance}
\end{align}
From the block diagonality of $\Sigma_{f,i}$ follows that equation \eqref{eq:state_variance} describes the convolution of the two CCU probability density functions of $e$ and $\tilde{x}$, which, according to Remark \ref{rem:ccu_property}, remains CCU. 
\begin{definition}[\hspace{1sp}\cite{b10} Probabilistic Reachable Set]
A set $\mathcal{R}$ is said to be a PRS of probability level $p$ for system \eqref{eq:extended_error} if
\begin{align*}
	\xi(0) = 0 \Rightarrow \text{Pr}(\xi(n) \in \mathcal{R}) \geq p \quad \forall n \geq 0.
\end{align*}
\end{definition}
In the following, we use Chebeyshev's inequality to construct PRS from the mean and variance information of the errors $\delta x, \delta u$. Since $\mathbb{E}(\delta x) = 0$, we get
\begin{align}
\mathcal{R}_i^X = \{ \delta x \in \mathbb{R}^{n_i} | \delta x^\top (\Sigma^X_{f,i})^{-1} \delta x \leq \tilde{p}_{i,x} \}, \label{eq:state_PRS}
\end{align}
where $p_{i,x} = 1 - n_i / \tilde{p}_{i,x}$ denotes the probability level. Similarly we can define the input PRS from the input covariance \eqref{eq:input_variance} as
\begin{align}
\mathcal{R}_i^U = \{ \delta u \in \mathbb{R}^{m_i}| \delta u^\top (\Sigma^U_{f,i})^{-1} \delta u \leq \tilde{p}_{i,u} \}. \label{eq:input_PRS}
\end{align}
\begin{remark}
The bound $\tilde{p}$ holds for arbitrary CCU distributions. However, if one knows the inverse cumulative density functions of $\mathcal{Q}^{w}_i$ and $ \mathcal{Q}^{d}_i$, the probability bound can be significantly tighter. In case of normal distributions, $\tilde{p} = \mathcal{X}_{n}^2(p)$ yields the tightest probabilistic bound for a sum of squared normals, where $\mathcal{X}_{n}^2(p)$ is the inverse cumulative density function of the Chi-squared distribution of $n$ degrees of freedom, evaluated at probability level $p$.
\end{remark}
\subsection{Constraint tightening}
In \eqref{eq:state_PRS} - \eqref{eq:input_PRS} we introduced ellipsoidal PRS. Hence, the constraint tightening cannot be done in standard form via Pontryagin set differences. However, by exploiting the marginalization property of CCU distributions, we are able to reformulate the ellipsoidal PRS into marginal PRS by using the marginal distribution in direction of each dimension of $\delta x$ and $\delta u$, i.e. the symmetric marginal PRS is given by
\begin{align*}
\mathcal{R}_{i,j}^X = \{ [\delta x]_j \in \mathbb{R} \: | \:  | [\delta x]_j | \leq \sqrt{\tilde{p}_{i,x} \cdot [\Sigma^X_{f,i}]_{j,j}}  \},
\end{align*}
where $\quad j = \{1, \ldots, n_i \}$.
 Similarly, we obtain the marginal PRS for the input distribution $\mathcal{R}_{i,k}^U, k = \{1, \ldots, m_i\}$. 
 Box shaped PRS are then simply given by the Cartesian products $\bar{\mathcal{R}}^X_i =\mathcal{R}_{i,1}^X \times \ldots \times \mathcal{R}_{i,n_i}^X$ and $\bar{\mathcal{R}}^U_i =\mathcal{R}_{i,1}^U \times \ldots \times \mathcal{R}_{i,m_i}^U$, which reduce the constraint tightening to the Pontryagin differences
\begin{align*}
	\mathbb{Z}_i = \mathbb{X}_i \ominus \bar{\mathcal{R}}_i^X, \quad \mathbb{V}_i = \mathbb{U}_i \ominus \bar{\mathcal{R}}_i^U.
\end{align*}
\begin{remark}
A less conservative constraint tightening may be achieved by considering the stationary distribution of $\Sigma_f^{\tilde{x}}$ and using a growing-tube inspired constraint tightening for $\Sigma^e(t|k)$, by computing a $t$-step PRS for the states and controls via \eqref{eq:covariance_prediction_LMI} and optimization problem \eqref{eq:optimization_variance}.
\end{remark}
The global constraint sets can now simply be obtained as the Cartesian products of the local sets, i.e.
\begin{align*}
\mathbb{Z} = \prod_{i \in \mathcal{M}} \mathbb{Z}_i, \quad \mathbb{V} = \prod_{i \in \mathbb{M}} \mathbb{V}_i.
\end{align*}
\subsection{Cost functions and distributed invariance}
In this work we consider a stabilizing MPC framework with terminal cost and terminal constraints. To this end, we make the following assumption:
\begin{assum}
\label{assum:distributed_ingredients}
There exists a terminal cost $V_f(z) = \sum_{i=1}^M V_{f,i}(z_i) = \sum_{i = 1}^M \Vert z_i \Vert_{P_i}^2 = \Vert z \Vert_P^2$ with block diagonal $P$, a distributed terminal controller $v = Kz = \text{col}_{i \in \mathcal{M}}(K_{\mathcal{N}_i} z_{\mathcal{N}_i})$ and a structured terminal set $\mathbb{Z}_f \subseteq \mathbb{Z}$, such that the following conditions hold for each $z \in \mathbb{Z}_f$
\begin{subequations}
\begin{align}
& V_f((A+BK) z) \leq V_f(z) - l(z,Kz), \label{eq:terminal_cost_dec} \\
& z \in \mathbb{Z}, \quad Kz \in \mathbb{V} \label{eq:terminal_constraints} \\
& (A+BK)z \in \mathbb{Z}_f. \label{eq:terminal_invariance}
\end{align}
\end{subequations}
The stage cost $l(z,v) = \sum_{i \in \mathcal{M}} (l_i(z_i, v_i))$ is the sum of local stage cost functions
\begin{align*}
l_i(z_i, v_i) = \Vert z_i \Vert_{Q_i}^2 +  \Vert v_i \Vert_{R_i}^2,
\end{align*}
where $Q_i\geq 0$, $R_i > 0$.
\end{assum}
\begin{remark}
\label{remark_terminal_invariance}
The design of a separable terminal cost function and distributed terminal controllers can be achieved via structured LMIs \cite{Conte1}. A structured terminal set $\mathbb{Z}_f$ is then defined as the largest feasible $\alpha$-level set of $V_f(z)$, i.e.
\begin{align*}
	\mathbb{Z}_f = \{ z \in \mathbb{R}^n | z^\top P z \leq \alpha \}, \quad \alpha \in \mathbb{R}_{>0},
\end{align*}
which can be solved efficiently as a distributed linear program, e.g. \cite[Sec 4.2]{Conte1} for details.
\end{remark}
For the MPC optimization problem we chose a finite horizon cost function   
\begin{align*}
	J(t|k) =  \sum_{i=1}^M \bigg \{ V_{f,i}(z_i(N|k)) + \sum_{t=0}^{N-1} l_i(z_i(t|k), v_i(t|k)) \bigg \},
\end{align*}
where $N$ denotes the prediction horizon.
\subsection{MPC optimization problem}
\label{subsection:MPC}
The following MPC optimization problem is solved via distributed optimization at every time instant $k \geq 0$ 
\begin{alignat}{2}
&\!\min_{\mathcal{V}, \mathcal{Z}}  &      &   \quad \scalebox{0.99}{$\displaystyle \sum_{i=1}^M \bigg \{ V_{f,i}(z_i(N|k)) + \sum_{t=0}^{N-1} l_i(z_i(t|k), v_i(t|k)) \bigg \} $} \label{mpc_problem} \\
&\text{s.t.} &      & \quad z(t+1|k) = A z(t|k) + B v(t|k), \hspace{1em} t=0,...,N-1 \nonumber\\
&                  &      & \quad (z(t|k),v(t|k)) \in {\mathbb{Z} \times \mathbb{V}}, \hspace{4.25em} t=1,...,N-1\nonumber\\
&                  &      & \quad z(N|k) \in {\mathbb{Z}_{f}}, \nonumber\\
&                  &      & \quad {z(0|k) \in  \{\hat{x}(k), z^*(1|k-1) \}}, \nonumber
\end{alignat}
where $\mathcal{V} = \{ v(0|k), \ldots, v(N-1|k) \}$ and $\mathcal{Z} = \{ z(0|k), \ldots, z(N|k) \}$ denote the input and state sequences. Each subsystem $i \in \mathcal{M}$ takes the first elements of the state and input sequences and implements them under the control law \eqref{eq:input_disturbed} to the real system \eqref{eq:global_system}. Then the remainder of the sequences are discarded, the new states are estimated and Problem \eqref{mpc_problem} is solved repeatedly with a shifted time window at time $k = k+1$.
\subsection*{Initial condition}
Before stating the main result of the paper, we briefly discuss the initial condition of Problem \ref{mpc_problem}. In stochastic MPC approaches with unbounded disturbances, recursive feasibility cannot be achieved by constraint tightening, e.g. as it is done in robust tube-based MPC. A straight forward approach to ensure this property is to initialize the optimization problem with the shifted optimal solution (Mode $2$), i.e. $z(0|k) = z^*(1|k-1)$, which leads to a poor closed-loop performance, since no feedback is applied. 
The second method (Mode $1$) is to initialize $z(0|k) = \hat{x}(k)$ with the disturbance affected state estimate. However, this can lead to infeasibility due to the unboundedness of the additive disturbance. 
To this end, we condition the initial state of Problem \ref{mpc_problem} on its feasibility in Mode 1 or Mode 2. Whenever Problem \ref{mpc_problem} is feasible in Mode $1$, then solve it, otherwise solve it in Mode $2$, which is guaranteed to be feasible. The following assumption is necessary to state the Lipschitz-based convergence result, which was similarly used in \cite{b10}.
\begin{assum}
The set $\Xi$ of feasible $z(0|k)$ in \eqref{mpc_problem} is bounded. \label{assum:bounded_set}
\end{assum}
\subsection*{Distributed ADMM}
In this paper, we use distributed consensus ADMM to solve Problem \ref{mpc_problem}. In \cite{Conte3}, the authors provided a corresponding formulation for distributed MPC, which we adopted in this paper. The algorithm asymptotically converges to the optimum of the original optimization problem \cite{boyd_admm}. Due to the linear convergence rate of ADMM, the algorithm achieves a medium accuracy within a few iterations, but for high accuracy an increasing number of iterations is necessary. In practice, this boils down to a trade-off between accuracy and computation time. For the sake of simplicity we make the following assumption.
\begin{assum}
\label{assum:exact_feasibility}
Problem \ref{mpc_problem} is solved exactly by distributed optimization.
\end{assum}
\begin{remark}
The assumption on exact minimization can be removed if one considers inexact minimization of Problem \ref{mpc_problem}, e.g. \cite{koehler2}, where an additional constraint tightening ensures feasibility of the uncertain predicted state trajectory.
\end{remark}
\begin{theorem}
\label{thm:main_result}
Let Assumptions \ref{assum_stabilizable}-\ref{assum:exact_feasibility} hold. If the MPC optimization Problem \ref{mpc_problem} admits a feasible solution at time $k=0$, then it is recursively feasible and the chance constraints \eqref{eq:chance_constraints} are satisfied in closed-loop for any $k \geq 0$ with convex symmetric PRS \eqref{eq:state_PRS} - \eqref{eq:input_PRS}. Furthermore, conditioned on $\delta x(0) = \delta u(0) = 0$, the controller achieves the following asymptotic average cost
\begin{align*}
	\lim_{r\rightarrow \infty} \frac{1}{r} \sum_{k = 0}^r \mathbb{E}(x^\top(k) Q x(k) + u^\top(k) R u(k)) \leq c,
\end{align*}
where $c = \gamma \sqrt{ \text{tr}(\Gamma^\top P_\Sigma \Gamma \Omega) } \geq 0$, $\gamma = \frac{\sqrt{2} \beta}{\sqrt{\lambda_{\text{min}}(P_\Sigma)}}>0$, $\beta$ denotes a Lipschitz constant and $P_\Sigma > 0$ the solution of the Lyapunov inequality $\Psi^\top P_{\Sigma} \Psi \leq P_{\Sigma} - \epsilon I$ for some $\epsilon > 0$.
\end{theorem}
The proof can be found in the appendix.
\section{Numerical example}
This section is dedicated to a brief numerical example. We consider $M=3$ subsystems with neighbors $\mathcal{N}_i = \{1,2,3\}, \forall i \in \mathcal{M}$, dynamic matrices  $A_{ii} = \left[ \begin{smallmatrix}
 	1 & 1\\
 	0 & 1
\end{smallmatrix} \right],
A_{ij} = \left[ \begin{smallmatrix}
	0.1 & 0\\
	0.1 & 0.1
\end{smallmatrix} \right], \forall j \in \mathcal{N}_i \backslash \{i\}, \forall i \in \mathcal{M}$, input matrices $
 B_{i} = \left[ \begin{smallmatrix}
  	0\\
 	1
\end{smallmatrix} \right], \forall i \in \mathcal{M}$ and output matrices $ C_{ii} = \left[ \begin{smallmatrix}
  	1 & 0.5
\end{smallmatrix} \right], \forall i \in \mathcal{M}$. Each subsystem is subject to a normally distributed process noise with $\Sigma_i^W = 0.005 I$ and a normally distributed measurement noise with $\Sigma_i^V = 0.001$. Each subsystem has to satisfy the chance constraint on the second state Pr$(-1 \leq [x_i]_2  \leq 0.5) \geq 0.6$. The weighting matrices are set to $Q_i = \left[ \begin{smallmatrix}
 	1 & 0\\
 	0 & 0.1
\end{smallmatrix} \right]$, $R_i = 0.1$ and the prediction horizon is $N = 15$. For simplicity, the terminal set is set to $\mathbb{Z}_{f,i} = \{0 \}$.
\subsection*{Simulation results}
 We carried out $K = 2000$ Monte-Carlo simulations with $10$ closed-loop steps, starting from the initial conditions $x_1(0) = [-6, 0]^\top$, $x_2(0) = [6, 0]^\top$ and $x_3(0) = [4, 0]^\top$. 
Figure \ref{fig:closed_loop} depicts the corresponding first $500$ closed-loop state trajectories, whereas in Figure \ref{fig:constraint_satisfaction} the point-wise in-time empirical constraint satisfaction is shown (for $K = 2000$). It can be seen that at every closed-loop instant $k \in \{ 1, ..., 10 \}$ the chance constraints are satisfied with the required level of $p_{i,x} = 0.6, \forall i \in \mathcal{M}$. 

\subsection*{Performance comparison}
Next, we compare the performance of a distributedly synthesized and centrally synthesized MPC setup. Hence, the system dynamics and MPC parameters remain unchanged, except for the terminal controller, injection gains and PRS synthesis. For the distributed setting we computed $K_d = \text{col}_{i \in \mathcal{M}} (K_{\mathcal{N}_i})$ and $L_d = \text{col}_{i \in \mathcal{M}}(L_{i})$ along Remark \ref{remark_stabilizable} via structured LMIs, whereas for the central setup we simply obtain the matrices $K_c$ and $L_c$ from the solution of the linear quadratic control and estimation problem.
\begin{center}
\begin{tabular}{lllll}
\cline{1-4}
\multicolumn{1}{|l|}{}      & \multicolumn{1}{l|}{av{[}$J^*${]}} & \multicolumn{1}{l|}{$\#C_{\text{vio}}$} & \multicolumn{1}{l|}{$\hat{p}(k)$} &  \\ \cline{1-4}
\multicolumn{1}{|l|}{Central}     & \multicolumn{1}{l|}{$17.6263$}             & \multicolumn{1}{l|}{$3039$}                                      & \multicolumn{1}{l|}{$0.799$}                                                        &  \\ \cline{1-4}
\multicolumn{1}{|l|}{Distributed} & \multicolumn{1}{l|}{$17.9883$}             & \multicolumn{1}{l|}{$2895$}                                      & \multicolumn{1}{l|}{$0.851$}                                                        &  \\ \cline{1-4}
                                  &                                    &                                                             &                                                                               & 
\end{tabular}
\vspace{-1em}
\captionof{table}{Comparison between central and distributed setups}
\label{table:results}
\end{center}
In Table \ref{table:results} we compare the average closed-loop cost $J^*$, the total number of constraint violations $\#C_{\text{vio}}$ and the smallest in-time empirical constraint satisfaction $\hat{p}(k)$. 

In both cases the probabilistic constraints were satisfied with the specified level, i.e. $\hat{p}(k) > p_{i,x} = 0.6$. It can be seen that the central setup produces a slightly lower average cost, which is the result of a less conservative chance constraint tightening due to the full knowledge of the state estimation vector $y$ and the fact that the injection matrix $L_c$ is dense. 

Furthermore, the central PRS are based on the exact stationary covariance matrix, while the distributed PRS use an over approximation, as stated in Remark \ref{rem:over_approx}. This yields a better exploitation of the chance constraints in the central setup, which can be seen by the higher total constraint violations $\#C_{\text{vio}}$ and lower $\hat{p}(k)$ in the central approach. In this example, the central PRS is about $36.1 \%$ smaller relative to the distributed PRS, which has a direct influence to the size of the feasible region of the MPC problem. 

As already stated in \cite{b10}, the strong closed-loop guarantees of the PRS-based SMPC approaches come at the price of a more conservative constraint tightening (empirical constraint satisfaction much larger than the required $p_x = 0.6$), which is furthermore amplified in the distributed setting.
\begin{figure}[h]
\centering
	  \includegraphics[width=0.9\linewidth]{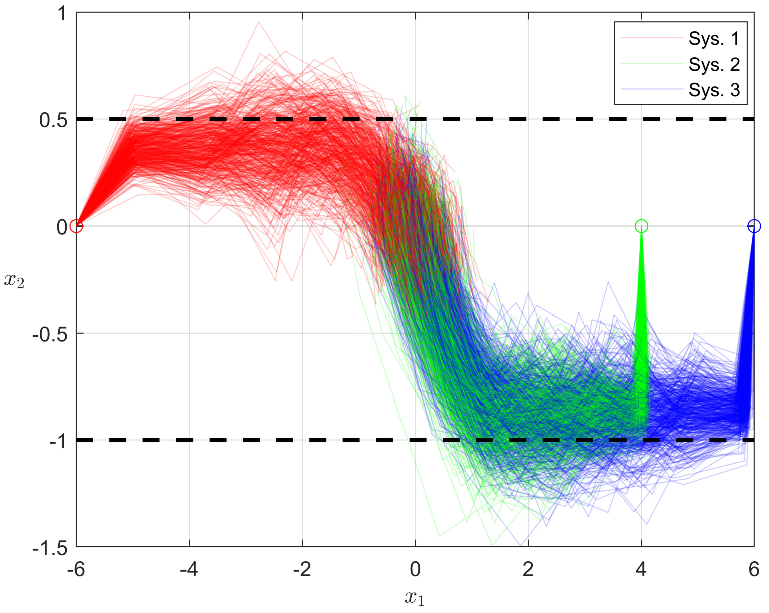}
	  \caption{$K = 500$ closed-loop trajectories for $10$ closed-loop time steps}
	  \label{fig:closed_loop}
\end{figure}
\begin{figure*}[t]
\centering
\begin{minipage}{.28\textwidth}
  \centering
  \includegraphics[width=0.9\linewidth]{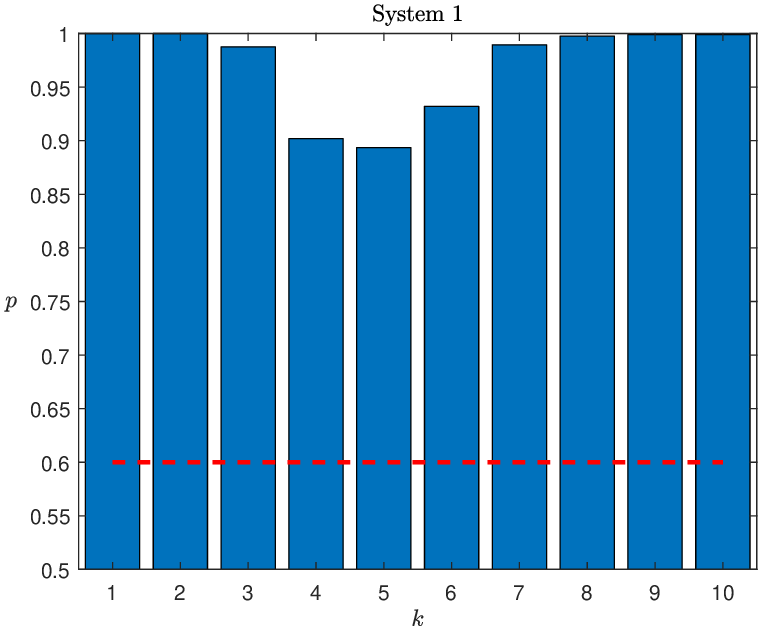}
\end{minipage}%
\begin{minipage}{.28\textwidth}
  \centering
  \includegraphics[width=0.9\linewidth]{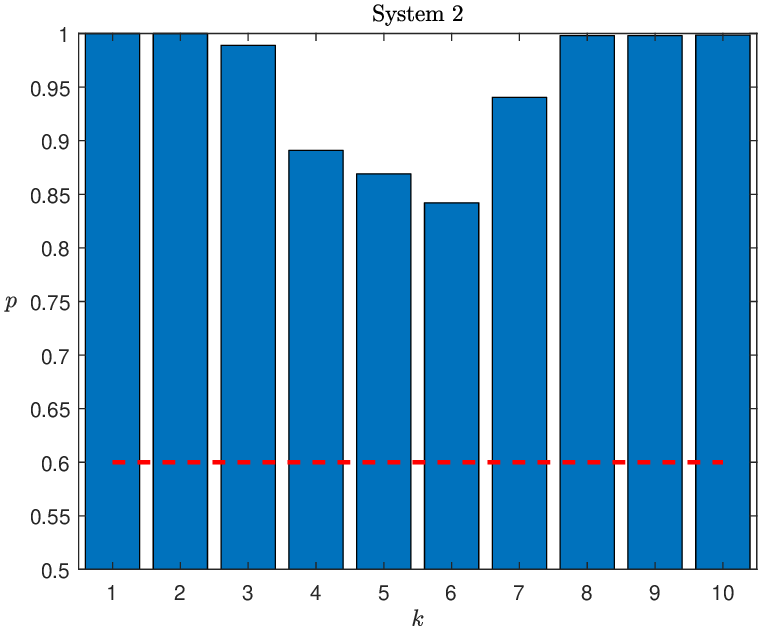}
\end{minipage}
\begin{minipage}{.28\textwidth}
  \centering
  \includegraphics[width=0.9\linewidth]{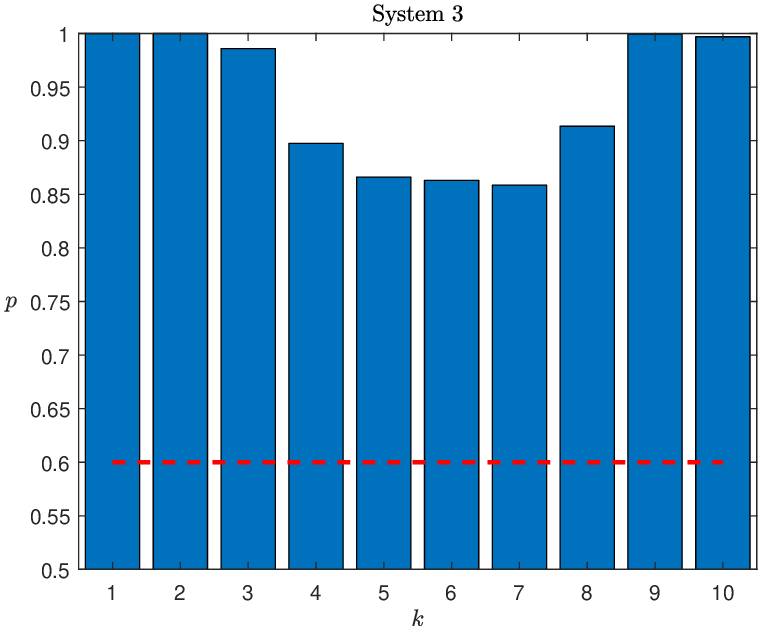}
\end{minipage}
\caption{Probability $(p)$ of constraint satisfaction for $k=1, ..., 10$. The red dotted line indicates $p_{i,x} = 0.6, i = 1, ..., 3$}
\label{fig:constraint_satisfaction}
\end{figure*}
\section{Conclusion}
We presented a stochastic output-feedback MPC scheme for distributed systems using Probabilistic Reachable Sets. The approach is highlighted through its fully distributed synthesis of the controller ingredients, the distributed PRS computation and the reduction to a quadratic program, which renders the optimization problem applicable for distributed ADMM. 
The optimization problem is proven to be recursive feasible, convergent to an average cost bound, while the chance constraint satisfaction is guaranteed for the closed-loop system. The numerical example reveals that the distributed PRS computation comes at the price of a more conservative chance constraint tightening, which results in a higher empirical chance constraint satisfaction rate than necessary.
\subsection*{Outlook}
Future work may include the investigation of the growing tube inspired approach for chance constraint tightening and as well as the inclusion of the inexact minimization framework, which makes the approach applicable to a wider range of practical problems. Another research direct may include the incorporation of coupling chance constraints.

\newpage
\appendix
\subsection{Proof of Theorem \ref{thm:main_result}}
The proof consists of four parts. First we show recursive feasibility and predictive chance constraint satisfaction. Then we show closed-loop chance constraint satisfaction, followed by the convergence proof. The last part in concerned about the asymptotic average cost bound. From the assumption on exact feasibility we can use the global vectors during the proof.

\textbf{Part 1: Recursive feasibility}
Consider that at time $k$ a feasible solution to Problem \ref{mpc_problem} exists. Then, at time $k+1$, we have to consider the possibly suboptimal solution due to Mode $2$. Lets define the shifted solutions 
\begin{align*}
&\scalebox{0.9}{$\tilde{z}(t|k+1) = [ z^*(1|k), ..., z^*(N|k), z(N+1|k)]$} \\
&\scalebox{0.9}{$\tilde{v}(t|k+1) = [ v^*(1|k), ..., v^*(N-1|k), A_K z^*(N|k)],$}
\end{align*}
where $z(N+1|k) = A_K z^*(N|k)$. In view of feasibility at time $k+1$ follows that $(\tilde{z}(t|k+1), \tilde{v}(t|k+1) ) \in (\mathbb{Z} \times \mathbb{V})$ for $t=0, ..., N-2$. For $t = N-1$ we have that $\tilde{z}(N-1|k+1) \in \mathbb{Z}_f$. 
Thus, by Assumption \ref{assum:distributed_ingredients}, and in particular from the invariance property \eqref{eq:terminal_invariance}, recursive feasibility follows. Predictive chance-constraint satisfaction is then a direct consequence, since for all $z \in \mathbb{Z}_f$ the terminal constraints \eqref{eq:terminal_constraints} are satisfied.

\textbf{Part 2: Closed-loop chance constraint satisfaction}. For brevity we show the closed-loop guarantees only for the state constraints. Consider the augmented error $\delta x = \tilde{x} + e$ with $\delta x(0|0) = \delta x(0) = 0$ and assume that $\mathcal{R}^X$ is a convex symmetric PRS. Now, at time $k + 1$, we condition the probability on feasibility of Problem \ref{mpc_problem} in Mode 1 or 2
\begin{align}
	&\text{Pr}(\delta x(k+1) \in \mathcal{R}^X) \nonumber \\ 
	= &\text{Pr}(\delta x(k+1) \in \mathcal{R}^X | M^1) \text{Pr}(M^1) \nonumber  \\
	+ &\text{Pr}(\delta x(k+1) \in \mathcal{R}^X | M^2)  \text{Pr}(M^2) \label{eq:conditional_prob}.
\end{align}
For Mode $2$ we have $z(0|k+1) = z(1|k)$, i.e.
\begin{align}
	\scalebox{0.95}{$\text{Pr}(\delta x(k+1) \in \mathcal{R}^X | M^2) =  \text{Pr}(e(1|k) + \tilde{x}(1|k) \in \mathcal{R}^X)$}. \label{eq:mode1_error}
\end{align}
For Mode $1$ we have $z(0|k+1) = \hat{x}(k+1)$, hence $e(k+1) = 0$ and the error evaluates according to
\begin{align*}
	&\text{Pr}(\delta x(k+1) \in \mathcal{R}^X | M^1) = \text{Pr}( \tilde{x}(k+1) \in \mathcal{R}^X) \\
	\geq &\text{Pr}( \tilde{x}(1|k) \in \mathcal{R}^X) \geq \text{Pr}( e(1|k) +  \tilde{x}(1|k) \in {\mathcal{R}}^X),
\end{align*}
where the first inequality follows from central convex unimodality of $\mathcal{Q}(0, \Sigma_{f}^{\tilde{x}})$ (Remark \ref{rem:ccu_property}) and \cite[Thm. 3]{b10}. The second inequality is due to \cite[Thm. 1]{anderson}. Substituting the latter inequality and \eqref{eq:mode1_error} into \eqref{eq:conditional_prob}, yields
\begin{align*}
	&\text{Pr}(\delta x(k+1) \in \mathcal{R}^X) \nonumber \\ 
	\geq &\text{Pr}(e(1|k) + \tilde{x}(k + 1) \in \mathcal{R}^X) \text{Pr}(M^1)  \nonumber \\
	+ & \text{Pr}( e(1|k) + \tilde{x}(k+1) \in \mathcal{R}^X) \text{Pr}(M^2)  \nonumber\\
	= & \text{Pr}( \delta x(1|k) \in \mathcal{R}^X),
\end{align*}
which, similar to \cite{b10}, bounds the closed-loop error. For further details, we refer the interested reader to the proof of \cite[Thm. 3]{b10}. Closed-loop chance constraint satisfaction is then a direct consequence of predictive chance constraint satisfaction.\\ 
\textbf{Part 3: Optimal cost decrease} The idea of the convergence proof is partially taken from \cite{b10}. Let $J(z(0|k), v(\cdot|k)) = \sum_{t=0}^{N-1} \Vert z(t|k) \Vert_Q^2 + \Vert v(t|k) \Vert_R^2 + \Vert z(N|k) \Vert_P^2$ be the optimal cost of Problem \ref{mpc_problem}. We condition the expected cost at time $k+1$ on feasibility of Problem \ref{mpc_problem} in Mode $1$ or Mode $2$
\begin{align}
	&\mathbb{E}(J^*(z(k+1)) \nonumber\\ 
	 = &\mathbb{E}(J^*(z(k+1))|M^2) \text{Pr}(M^2) \nonumber \\
	+ &\mathbb{E}(J^*(z(k+1))|M^1) \text{Pr}(M^1). \label{eq:cost_decrease_expected}
\end{align} 
The first term directly satisfies
\begin{align}
\mathbb{E}(J^*(z(k+1))|M^2) \leq J(z(1|k), \tilde{v}(\cdot|k+1)), \label{eq:mode1_decrease}
\end{align}
where $\tilde{v}(\cdot|k+1)$ is the shifted control sequence. According to \cite{bemporad} the optimal cost $J^*(z)$ of a nominal MPC problem is piecewise quadratic in $z$, then by Assumption \ref{assum:bounded_set} follows that there exists a Lipschitz constant $\beta$, such that  
\begin{align}
	J^*(z + \delta x) \leq J^*(z) + \beta \Vert \delta x \Vert_2 .\label{eq:bound_cost}
\end{align}
The expected value for Mode $1$ is evaluated according to
\begin{align*}
	&\mathbb{E}( J^*(\hat{x}(k+1))\vert M^1) = \mathbb{E}(J^*(z(0|k+1))|M^1) \\
	&\overset{\eqref{eq:bound_cost}}{\leq} J(z(1|k), \tilde{v}(\cdot|k+1)) + \beta \mathbb{E} \big(\Vert x(k+1) - z(1|k) \Vert_2 \vert M^1 \big ).
\end{align*}
Now we can add $\beta\mathbb{E}(\Vert x(k+1) - z(1|k) \Vert_2 \vert M^2 )$ to \eqref{eq:mode1_decrease} and substitute both inequalities into \eqref{eq:cost_decrease_expected}, which yields
\begin{align*}
		\mathbb{E}(J^*(z(0|k+1)) &\leq  J(z(1|k), \tilde{v}(\cdot|k+1)) \\
		 &+ \beta \mathbb{E}\big(\Vert x(k+1) - z(1|k) \Vert_2\big ).
\end{align*} 
The latter term can be further evaluated by considering the decomposition $x(k+1) - z(1|k) = \tilde{x}(1|k) + e(1|k) = [I \quad I] \: \xi(1|k)$, i.e.
\begin{align}
&\beta \mathbb{E}\big(\Vert x(k+1) - z(1|k) \Vert_2\big ) = \beta\mathbb{E}\big(\Vert [I \quad I] \: \xi(1|k) \Vert_2  \big ) \nonumber \\
&\leq \sqrt{2} \beta\mathbb{E}\big(\Vert \: \xi(1|k) \Vert_2  \big ) \nonumber \\
&\overset{\eqref{eq:extended_error}}{=} \sqrt{2} \beta\mathbb{E} \bigg (  \Vert \Psi \xi(0|k) \Vert_2+ \Vert \Gamma \omega(0|k) \Vert_{2} \bigg ) \nonumber \\
&\leq \underbrace{\frac{\sqrt{2} \beta}{\sqrt{\lambda_{\text{min}}(P_\Sigma)}}}_{\gamma} \bigg (  \Vert \Psi \xi(0|k) \Vert_{P_\Sigma}  + \mathbb{E}\big( \Vert \Gamma \omega(0|k) \Vert_{P_\Sigma} \big ) \bigg ) \label{eq:proof_cost_bound}
\end{align}
where the first inequality is due to the triangle inequality together with the global version of \eqref{eq:extended_error}. The second inequality uses $\sqrt{\lambda_{min}(P_{\Sigma})} \Vert \xi \Vert_2 \leq \Vert \xi \Vert_{P_{{\Sigma}}}$, where $P_\Sigma > 0$ denotes the solution of the Lyapunov inequality
\begin{align*}
	 \Vert \Psi \xi(0|k) \Vert_{P_{\Sigma}}  \leq  \Vert \xi(0|k) \Vert_{P_{\Sigma}} - \epsilon \Vert \xi(0|k) \Vert_{P_{\Sigma}}
\end{align*}
for some $\epsilon>0$. Thus, \eqref{eq:proof_cost_bound} can be bounded by
\begin{multline*}
\beta \mathbb{E}\big(\Vert x(k+1) - z(1|k) \Vert_2\big )  \\
\leq \gamma \bigg (  (1-\epsilon) \Vert \xi(0|k) \Vert_{P_{\Sigma}} + \mathbb{E}\big( \Vert \Gamma \omega(0|k) \Vert_{P_\Sigma} \big ) \bigg ). 
\end{multline*}
If we combine the latter inequality with the nominal MPC cost decrease due to the terminal controller \eqref{eq:terminal_cost_dec}, we obtain
\begin{align*}
&\mathbb{E}\big(J^{*}\big(z(k+1)\big) \big) - J^*\big(z(k) \big) \\
	&\leq - \Vert z(k) \Vert^2_Q - \Vert v(k) \Vert^2_R - \gamma \epsilon \Vert \xi(k) \Vert_{P_{\Sigma}} + \gamma \mathbb{E}(\Vert \Gamma \omega(k) \Vert_{P_{\Sigma}}
\end{align*}
where $z(k) = z(0|k)$ and $\Omega = \text{blkdiag}(\Sigma^W, \Sigma^D)$. \\
\textbf{Part 4: Asymptotic average cost bound}\\
 Using standard arguments from stochastic control, we obtain
\begin{align*}
	0 \leq&\lim_{r \rightarrow \infty} \frac{1}{r} \mathbb{E}\big(J^{*}\big(z(k)\big) \big) - J^*\big(z(0) \big) \\
	\leq &\lim_{r \rightarrow \infty}  \mathbb{E}\bigg(\sum_{l = 0}^{r}  - \Vert z(l) \Vert^2_Q - \Vert v(l) \Vert^2_R - \gamma \epsilon \Vert \xi(l) \Vert_{P_{\Sigma}} \\
	& \hspace{3.7em}+ \gamma \mathbb{E}(\Vert \Gamma \omega(l) \Vert_{P_{\Sigma}} \bigg ) \\
	\leq &\lim_{k \rightarrow \infty}  \sum_{l = 0}^{k} \gamma \mathbb{E}(\Vert \Gamma \omega(l) \Vert_{P_{\Sigma}} = \gamma \sqrt{\text{tr}(\Gamma^\top P_\Sigma \Gamma \Omega) } = c,
\end{align*}
 which concludes the proof. \qed
\vspace{12pt}
\end{document}